\documentclass[10pt]{amsart}
\usepackage{graphicx}
\allowdisplaybreaks
\usepackage{amsmath,amssymb,mathrsfs,amsthm}

\usepackage{verbatim}
\usepackage[spanish,USenglish]{babel} % espanol, ingles
\usepackage{color}
\usepackage{tikz}
\usepackage{graphicx}

\newtheorem{thm}{Theorem}[section]
\newtheorem{corollary}[thm]{Corollary}
\newtheorem{proposition}[thm]{Proposition}
\newtheorem{theorem}[thm]{Theorem}
\newtheorem{lemma}[thm]{Lemma}

\theoremstyle{definition}
\newtheorem{definition}[thm]{Definition}

\theoremstyle{remark}
\newtheorem{remark}[thm]{Remark}

\numberwithin{equation}{section}

\def\N{\mathbb{N}}
\def\R{\mathbb{R}}

\def\Z{\mathbb{Z}}

\def\Tn2{{\mathcal T}^{(n)}_2}

\def\w{\omega}

\begin{document}
	\title[WELL-POSEDNESS FOR CAUCHY FRACTIONAL PROBLEMS]{WELL-POSEDNESS FOR CAUCHY FRACTIONAL PROBLEMS INVOLVING DISCRETE CONVOLUTION OPERATORS}

	\author[Jorge Gonz\'alez-Camus]{Jorge Gonz\'alez-Camus}
	\address{Departamento de Matem\'atica, Facultad de Ciencias Naturales, Matemática y del Medio Ambiente, Universidad Tecnológica Metropolitana, Santiago, Chile.}
	\email{j.gonzalezc@utem.cl}

	\thanks{}
	
	

	\maketitle
	
	\begin{abstract} This work is focused on establishing sufficient conditions to guarantee the well-posedness of the following nonlinear fractional semidiscrete model  
		\begin{equation*}
			\begin{cases}  \mathbb D^\beta_t u(n,t)= B u(n,t) + f(n-ct,u(n,t)),\, &n\in\mathbb{Z}, \;t>0,\\
				u(n,0)=\varphi(n),\; &n\in\mathbb{Z},\\
			\end{cases}
		\end{equation*}
		under the assumptions that $\beta \in (0,1]$, $c>0$ some constant, $B$ is a discrete convolution operator with kernel $b\in\ell^1(\Z)$, which is the infinitesimal generator of the Markovian $C_0$-semigroup and suitable nonlinearity $f$. We present results concerning the existence and uniqueness of solution, as well as establishing a comparison principle of solutions according to respective initial values.
	\end{abstract}

	\section{Introduction}
	
	This work is devoted to the study of fractional nonlinear models whose spatial variable lies in the one dimensional mesh $\mathbb{Z}$ (also called one dimensional infinite lattice)
	\begin{equation}\label{Mainequation}
		\begin{cases}  \mathbb D^\beta_t u(n,t)= B u(n,t) + f(n-ct,u(n,t)), &n\in\mathbb{Z}, \;t>0,\\
			u(n,0)=\varphi(n),\; &n\in\mathbb{Z},\\
		\end{cases}
	\end{equation}
	where $\beta \in (0,1]$ is a real number, $c>0$ some positive constant, the operator $Bh(n)=(b \ast h)(n)$, with $b$ belonging to Banach Algebra $\ell^{1}(\mathbb{Z})$ being an infinitesimal generator of a Markovian semigroup and $f$ and $\varphi$  suitable functions, $p\in [1,\infty]$ and $\mathbb{D}_t^{\beta}$ denotes the \emph{Caputo fractional derivative} 
	\begin{equation*}\label{caputo}
		\mathbb{D}_t^{\beta}v(t)=\frac{1}{\Gamma(1-\beta)}\int_{0}^{t}(t-s)^{-\beta}v^{\prime}(s)ds,
	\end{equation*}
	for $t>0$ and $\alpha>0$, and $v$ a differentiable function.
	
	Firstly, we consider the linear version of \eqref{Mainequation}
	\begin{equation}\label{Main1}
		\begin{cases}  \mathbb D^\beta_t u(n,t)= B u(n,t) + f(n,t), &n\in\mathbb{Z}, \;t>0,\\
			u(n,0)=\varphi(n),\; &n\in\mathbb{Z},\\
		\end{cases}
	\end{equation}
	where $Bf=b\ast f$ is a discrete convolution operator over $\ell^p(\Z), \ p\in [1,\infty]$.  This system generalizes several interesting cases that can be found in the literature. For instance, for the \emph{discrete fractional laplacian}, denoted by $B=(-\Delta_d)^{\alpha}$, we obtain several semidiscrete models:
	\begin{itemize}
		\item For $\alpha=\beta=1$ we get the \emph{semidiscrete heat equation} (see \cite{CGRTV}).
		\item For $\alpha=\frac{1}{2},\beta=1$ we get the \emph{semidiscrete Poisson equation} (see \cite{CGRTV}).
		\item For $\alpha=1,\beta=\frac{1}{3}$ we get the \emph{semidiscrete Airy equation} (see \cite{VaSo2010}). 
	\end{itemize}  
	Similarly, for the \emph{fractional backward/forward Euler operator}, denoted by $B=(-\nabla)^{\alpha}/\Delta^{\alpha}$, for $\alpha=\beta=1$, we get the linear \emph{fractional semidiscrete transport equation} (see \cite{ALT}).
	Both operators have received special attention, due to the numerous applications in which they are involved, mainly diffusive and transport models. We refer recent references \cite{ALT, CGRTV, GKLW, GLM, KLW3, LiRo}.
	
	Semidiscrete equations have become an important focus of study in the last decade, due to their many applications in diverse areas of knowledge. In the literature many examples can be found. For instance, in chemistry, \eqref{Main1} describes the flow of a chemical in an infinite system of tanks arranged in a row with neighboring tanks connected by a pair of pipes \cite{Sl13}. In probability theory, for the case $B =\Delta_d$, the value, for $u$ a discrete harmonic function, $u(n,t)$  describes the movement of a particle that jumps either to the adjacent left point or to the adjacent right point with probability $\frac{1}{2}$ \cite{Stinga}. In transport theory, \eqref{Main1} describes the dynamics of an infinite chain of cars, each of them being coupled to its two neighbors, where the value $u(n,t)$ is the displacement of car $n$ at time $t$ from its equilibrium position \cite{FB12}. This topic has been studied by the branch of mathematical analysis, using techniques from PDEs, functional and harmonic analysis, among others. We refer to \cite{CLRV, GKLW, GLM, KLW1, KLW2, KLW3, LiMu, Tasarov, Tasarov2, Tasarov3} and references therein.

	The study of semidiscrete models is not new. To the best of our knowledge, the pioneering work on this topic was presented by Bateman in \cite{Bateman} in 1943, where the author considered the Cauchy problem of first and second order involving a broad set of discrete operators. The author presented a list of $C_0$-semigroup and cosine operators associated (in the paper called "influence functions") using the method of generating functions, with application to a wide array of physical and engineering systems. The list presented shows us the crucial role played by the special functions, such as the Bessel type, Hermite's polynomial, among others. Later, many researchers focused their efforts on the study of properties of the difference operators. For instance, for diffusive and transport models, we refer to \cite{AGMP, Slavik1, Slavik2, Slavik3, Sthelik, Volek}.

	The analytical study of this class of equations and the representation of fundamental solution (also known as "Green functions") and its properties has received increasing interest from researchers in the last decade. Ciaurri et. al in \cite{CGRTV} obtain the maximum principle, weighted $\ell^p(\mathbb{Z})$-boundedness of conjugate harmonic functions, Riesz transforms and square functions of Littlewood-Paley for \eqref{Main1} considering $ B=(-\Delta)^\alpha, \alpha\in(0,1]$ and $\beta=1$. Under the same assumptions, Lizama and Roncal proved properties of the fundamental solution, such as spectrum, $\ell^1$-norm, almost periodicity, etc, using operator theory techniques with the properties of the Bessel functions to develop a theory of analytic semigroups and cosine operators generated by this operator. With regard to fractional semidiscrete systems, Keyantuo et. al. in \cite{KLW3} determined the explicit integral representation of the fundamental solution for \eqref{Main1} for $B=(-\Delta_d)^{\alpha}$ in the subdiffusive case (i.e, $0<\alpha,\beta\leq1$) through the subordination principles on Wright and L\'evy functions (see \cite{AM}). Furthermore, the authors proved the fact that the discrete fractional Laplacian generates a Markovian semigroup (see \cite[Theorem 3.8]{KLW3}). After, González-Camus et. al. in \cite{GKLW} got the representation of the explicit fundamental solution for \eqref{Main1} in the subdiffusive and superdiffusive fractional cases (i.e, $0<\alpha\leq1$, with  $0<\beta\leq1$ and $1<\beta\leq2$ respectively) in terms of the generalized Wright functions. It is worthwhile to mention that the discrete Fourier transform is the fundamental tool in order to get their respective representations and its properties.
	
	For the fractional backward/forward Euler operator, in \cite{ALT} the authors present a maximum and comparison principles for the discrete fractional derivatives, as well as regularity results when the space is a one dimensional mesh of length $h$. Rates of approximation results to the continuous fractional derivative in Marchaud sense are also accomplished.
	
	Quite recently, a general method for the study of the system \eqref{Main1} was presented by González-Camus et. al in \cite{GLM}, where the solution (unique) was obtained explicitly in terms of Mittag-Leffler on the Banach Algebra $\ell^1(\Z)$, as well as an integral representation and its properties such as $\ell^1(\mathbb{Z})$-norms, spectrum, among others, based on the subordination principle on Wright function and functional calculus presented in \cite{AM} and \cite{Ba2001}, respectively. The explicit representation through Mittag-Leffler function is given by
	\begin{align*}\label{solutionlinearcase}
		\displaystyle u(n,t)= (E_{\beta, 1}(t^\beta b)\ast\varphi)(n) +\int_{0}^{t}(t-s)^{\beta-1}\left(E_{\beta, \beta}((t-s)^\beta b)\ast f(\cdot,s)\right)(n)\ ds, \,\,  n\in \Z.
	\end{align*}
	Furthermore, the authors proved that if the initial data $\varphi\in \ell^p(\Z)$, for $1<p\leq\infty$ and $g$ is a suitable function, then this solution $u$ belongs to $\ell^p(\Z)$.

	Nonlinear models, and particularly nonlinear semidiscrete systems, have been the subject of research of many authors, due to their relevance in several areas as, physics, chemistry, biology, among others. For instance, the Fisher equation arises as a deterministic version of a stochastisc model for the spatial spreed of a favoured gene in a population, which is given by
	\begin{equation*}
		u_t(n,t)=\Delta_d u(n,t)+ ru(n,t)-u^2(n,t), \ n \in \Z, \ t \geq 0.
	\end{equation*}
	This equation was considered by Zinner, Harris and Hudson in \cite{Zinner}, where the authors focus their study on the existence and nonexistence of traveling wavefronts. A retarded version was studied by Hu and Li in \cite{Hu}
	\begin{equation*}
		u_t(n,t)=\Delta_d u(n,t)+ r(n-ct)u(n,t)-u^2(n,t), \ n \in \Z, \ t \geq 0, \ c>0,
	\end{equation*}
	where the authors proved that the long term behavior of solutions depends on the speed of the shifting habitat edge $c$. Another interesting example is provided by the \emph{discrete Nagumo equation}
	\begin{equation*}
		u_t(n,t)=\rho\Delta_d u(n,t)+ u(n,t)(u(n,t)-a)(1-u(n,t)), \ n \in \Z, \ t \geq 0,
	\end{equation*}
	where $\rho>0$ and $a \in (0,\frac{1}{2})$. The discrete Nagumo equation has been the focus of attention due to its application in models for conduction in myelanited axon. For more information we refer \cite{Bell, Keener, KLW3}.
	
	Due to the above examples of nonlinear models, Keyantuo et. al. in \cite{KLW3} studied the following model
	\begin{equation*}
		\begin{cases}  \mathbb D^\beta_t u(n,t)= -(\Delta_d)^{\alpha} u(n,t) + f(n-ct,u(n,t)), &n\in\mathbb{Z}, \;t>0,\\
			u(n,0)=\varphi(n),\; &n\in\mathbb{Z},\\
		\end{cases}
	\end{equation*}
	considering a suitable function $f$ and $0<\alpha\leq 1$, obtaining general results including all equations presented above. 
	
	We can notice that all the above nonlinear models are governed by the spatial operator discrete fractional Laplacian. Here arises a natural question: is it an intrinsic property of the discrete fractional Laplacian, or is there a deeper concept which is crucial in order to obtain this well-posedness? The main objective of this study is to answer the question above.
	Our purpose in this paper is to provide sufficient conditions for the well-posedness of the nonlinear model \eqref{Mainequation}, generalizing the results obtained in \cite{KLW3} for a branch of important difference operators and its fractional powers. To answer this question, we will use tools from Banach Algebra $\ell^1(\Z)$, resolvent families of operators theory and properties of Markov operators.

	This paper is organized as follows: In Section 2, we present the basic concepts of this work, focusing mainly on the Banach Algebra $\ell^1({\Z})$, properties of discrete convolution operators, Markovian semigroups and resolvent operators theory. In Section 3 we present the main result of this study in the Theorem \eqref{MainTheorem}. We give the sufficient conditions to ensure the well-posedness of the nonlinear model \eqref{Mainequation} and uniqueness of solution, through an iterative method and Gronwall inequality. Finally, we present two corollaries concerning the comparison principle for solutions according to respective initial values. In Section 4 we prove the Markovian property for $C_0$-semigroup generated by fractional powers of a discrete convolution operator which is an infinitesimal generator of a Markovian $C_0$-semigroup, concluding with applications of the main result. We finish this paper with the Appendix, including the definitions and properties about special functions used throughout this study.

	\section{Preliminaries}
	
	In this section, we give the basic concepts and results concerning to discrete convolution operator. Furthermore, we present the notation used throughout this work.
	
	\subsection{Operators defined by discrete convolution}
	
	We consider as framework of this study the commutative \emph{Banach Algebra with unit}
	\begin{equation*}
		\ell^1(\Z):=\{f:\Z \to \R:\sum_{n \in \Z}|f(n)|<\infty\},
	\end{equation*}
	equipped with the product the \emph{discrete convolution} defined by
	\begin{equation*}\label{convolutiondef}
		(f\ast g)(n):=\sum_{j\in\mathbb{Z}}^{} f(n-j)g(j), \qquad n\in \Z.
	\end{equation*}
	The unit for the product is given by \emph{$\delta$-Kronecker measure} concentrated at $0$, denoted by $\delta_0$. This element satifies: for all  $a\in\ell^1(\mathbb{Z}),\ a\ast \delta_0=a$. It is convenient recall that the $\delta$-Kronecker measure $\delta_n$ is defined by $$\delta_n(m):=\begin{cases}
		1, & \quad m=n,\\
		0, & \quad m\neq n.
	\end{cases}$$
	
	Furthermore, we remark that  $\ell^1(\Z)$ is semi-simple and regular Algebra (see \cite[Theorem 4.7.4]{La}) and \cite[Corolary 7.2.3]{La} respectively).
	
	Given $a=(a(n))_{n\in \Z}\in \ell^1(\Z)$,  we define the following discrete operator by convolution
	\begin{equation*}\label{convolutiondefinition}
		Af(n):= (a\ast f)(n), \ {n\in \Z}, \ f\in \ell^p(\Z).
	\end{equation*} 
	We say that $A$ is a \emph{discrete convolution operator}, with \emph{kernel} $a$. We note that $A\in{\mathcal B}(\ell^p(\Z)):=\{B:\ell^p(\Z)\to\ell^p(\Z):B\ \mbox{is linear and bounded}\}$ for all $1\le p\le \infty$, since from Young's inequality $$\|Af(n)\|_p=\|(a\ast f)(n)\|_p\leq \|a\|_1\|f\|_p .$$ Moreover, the operator $A$ satisfies 
	\begin{equation}\label{Norms}
		\Vert A\Vert_1=\Vert a\Vert_1.
	\end{equation}
	Here, we understand the powers for $a \in \ell^{1}(\mathbb{Z})$ through discrete convolution, given by
	\begin{equation*}
		a^j:=a^j(n):=(\underbrace{a\ast...\ast a}_{j-times})(n).
	\end{equation*}
	Obviously, we have $a^{0}=\delta_0$, $a^{1}= a$.
	
	In the Banach Algebra $\ell^1(\Z)$, given $\alpha, \beta >0,$ we define the vector-valued \emph{Mittag-Leffler function} $E_{\alpha,\beta}: \ell^1(\Z)\to \ell^1(\Z)$ as follows
	$$
	E_{\alpha, \beta}(a) := \sum_{j=0}^{\infty}\frac{a^{ j}}{\Gamma(\alpha j + \beta)}, \quad a\in\ell^1(\Z).
	$$
	We define $e^{at}$ \emph{the exponential of $a$} as follows
	\begin{equation*}\label{expo}
		e^{t a}:=\sum_{n=0}^\infty{t^na^{n}\over n!}, \ t\in \R_0^+.
	\end{equation*}
	Note that $\displaystyle E_{1,1}(ta)=\sum_{j=0}^\infty \frac{(ta)^{ j}}{j!}=e^{ta}.$
	
	We denote $\{e^{tA}\}_{t\geq 0}$ to the \emph{family semigroup of operators} generated by the operator $A\in \mathcal{B}(\ell^p(\Z))$. In the case of discrete convolution operators, $Af=a\ast f$, this family acts in fact by discrete convolution over $\ell^p(\Z)$. Moreover, we have the explicit representation
	\begin{equation*}\label{exponentionoperator}
		e^{tA}f(n)=(e^{at}\ast f)(n),\ n \in \mathbb{Z}, \ f \in \ell^p(\Z),
	\end{equation*}
	and, for each $t\geq 0$, from the equality \eqref{Norms} we have
	\begin{equation}\label{normsemigroupskernel}
		\|e^{tA}\|_1=\|e^{ta}\|_1.
	\end{equation} 
	See \cite{GLM} and \cite{La} for more information about the Banach Algebra $\ell^1(\Z)$ and results for discrete convolution operators. We refer \cite{Engel, Pazy} for more information about the semigroup theory.
	
	\subsection{Fractional Difference Operators}
	
	Next, we present three important operators.
	
	\subsubsection{The discrete Laplacian} Denoted by $\Delta_d$ given by
	\begin{equation*}\label{discretelaplacian}
		\Delta_d f(n):= f(n+1)-2f(n)+f(n-1)=([\delta_{-1}-2\delta_0+\delta_{1}]\ast f)(n).
	\end{equation*}
	For $\alpha \in (0,1)$ the \emph{discrete fractional Laplacian} is given by
	\begin{align*}\label{discretefractionallaplacian}
		(-\Delta_d)^\alpha f(n)=\sum_{k\in\mathbb{Z}}K_d^\alpha(n-k)f(k)=(K_d^\alpha\ast f)(n),\, n\in\mathbb{Z},
	\end{align*}
	where the coefficients $K_d^\alpha$ are given by
	\begin{equation*}\label{kernelfractionallaplacian}
		K_d^\alpha(n)=\frac{1}{2\pi}\int_{-\pi}^\pi (4\sin^2(\theta/2))^\alpha e^{-in\theta}d\theta=\frac{(-1)^n\Gamma(2\alpha+1)}{\Gamma(1+\alpha+n)\Gamma(1+\alpha-n)},\, n\in\mathbb{Z}.
	\end{equation*}
	We mention that this operator corresponds to fractional powers of the second-order central difference approximation
	for the second-order derivative. For more information about the representation and properties of the discrete fractional Laplacian, see \cite{GKLW}, \cite{GLM} and \cite{KLW3}.
	\subsubsection{The forward/backward Euler operators} Denoted by $-\Delta$ and $\nabla$ respectively, are given by
	\begin{equation*}\label{discreteright}
		-\Delta f(n):=f(n)-f(n+1)=([\delta_0-\delta_{-1}]\ast f)(n).
	\end{equation*}
	\begin{equation*}\label{discreteleft}
		\nabla f(n):=f(n)-f(n-1)=([\delta_0-\delta_{1}]\ast f)(n).
	\end{equation*}
	From the above definition, $\alpha \in (0,1)$ the \emph{fractional forward/backward Euler operator}, given by
	\begin{align*}\label{discretefractionalright/left}
		(\Delta/-\nabla)^\alpha f(n)=\sum_{k\in\mathbb{Z}}K_{\pm}^\alpha(n-k)f(k)=(K_{\pm}^\alpha\ast f)(n),\, n\in\mathbb{Z},
	\end{align*}
	where  the coefficients $K_{\pm}^\alpha$  are given by
	\begin{equation*}\label{kernelfractionalright}
		K_{\pm}^\alpha(n)=\frac{1}{2\pi}\int_{-\pi}^\pi (1-e^{\pm i\theta})^\alpha e^{-in\theta}d\theta=(-1)^n\frac{\Gamma(\alpha+1)}{\Gamma(\alpha \mp n+1)\Gamma(\pm n+1)}\chi_{\mathbb{Z}_{\pm}}(n),
	\end{equation*}
	where $n\in\mathbb{Z}$ and $\mathbb{Z}_{\pm}:=\{\ldots,\pm2,\pm1,0\}$. It is convenient remark that the kernels $K_{+}^\alpha$ and $K_{-}^\alpha$ correspond to $\Delta^\alpha$ and $(-\nabla)^\alpha$ respectively.
	We recall that the above operators are related to fractional powers of the Euler scheme of approximation. For more information about the representation and properties of the discrete fractional differences see \cite{ALT} and \cite{GLM}.
	\subsubsection{The discrete Laplacian 2-step} Denoted by $\Delta_{dd}$ given by
	\begin{equation*}\label{discretelaplacian}
		\Delta_{dd} f(n):= f(n+2)-2f(n)+f(n-2)=([\delta_{-2}-2\delta_0+\delta_{2}]\ast f)(n).
	\end{equation*}
	For $\alpha \in (0,1)$ we obtain the \emph{discrete fractional Laplacian 2-step}, given by
	\begin{align*}\label{discretefractionallaplacian2}
		(-\Delta_{dd})^{\alpha} f(n)=\sum_{k\in\mathbb{Z}}K_{dd}^{\alpha}(n-k)f(k)=(K_{dd}^{\alpha}\ast f)(n),\, n\in\mathbb{Z},
	\end{align*}
	where the coefficients $K_{dd}^{\alpha}$ are given by
	\begin{equation*}\label{kernelfractionallaplacian2}
		K_{dd}^{\alpha}(n)=\frac{1}{2\pi}\int_{-\pi}^\pi (4\sin^2(\theta))^\alpha e^{-in\theta}d\theta=\frac{\cos({n\over 2}\pi)\Gamma(2\alpha+1)}{\Gamma(1+\alpha+{n\over 2})\Gamma(1+\alpha-{n\over 2})}, 
	\end{equation*}
	with $n\in\mathbb{Z}$. This operator appears in De Juhasz equation, appearing in the seminal Bateman’s paper \cite{Bateman} in connection with surges in springs and connected systems of springs. For more information about the representation and properties of the discrete fractional Laplacian 2 step, see \cite{GLM}.
	
	The remarkable fact is, for $\alpha \in (0,1)$, $(-\Delta_d)^{\alpha}$, $(\Delta/-\nabla)^\alpha$ and $(-\Delta_{dd})^{\alpha}$ are linear bounded operators over $\ell^p(\mathbb{Z}), 0\leq p\leq\infty$, since for $|n|\to \infty$, $|K_d^{\alpha}(n)|\sim \frac{1}{n^{2\alpha+1}}$,  $|K_{\pm}^{\alpha}(n)|\sim \frac{1}{n^{\alpha+1}}$ and $|K_{dd}^{\alpha}(n)|\sim \frac{1}{n^{2\alpha+1}}$ respectively. See  \cite[Section 4]{GLM}, \cite[eq. 4.5, section 4.2]{Ortigueira} for more information.

	It is well-known that the semigroups generated by $\Delta_d$, $-\Delta/\nabla$ and $\Delta_{dd}$ are (see \cite{Bateman}, \cite{GLM})
	\begin{equation*}
		e^{t\Delta_d}f(n)=\sum_{m\in\mathbb{Z}}^{}e^{-2t}I_{n-m}(2t)f(m)=:(e^{t[\delta_{-1}-2\delta_0+\delta_{1}]}\ast f)(n).
	\end{equation*}
	\begin{equation*}
		e^{t(-\Delta/\nabla)}f(n)=\sum_{m\in\mathbb{Z}}^{}e^{-t}\frac{t^{\mp(n-m)}}{(\mp(n-m))!}\chi_{\Z_{\mp}}(n-m)f(m)=:(e^{t([\delta_0-\delta_{-1}]/[\delta_0-\delta_{1}])}\ast f)(n).
	\end{equation*}
	\begin{equation*}
		e^{t\Delta_{dd}}f(n)=\sum_{m\in\mathbb{Z}}^{}e^{-2t}I_{\frac{n-m}{2}}(2t)\chi_{2\mathbb{Z}}(n-m)f(m)=:(e^{t[\delta_{-2}-2\delta_0+\delta_{2}]}\ast f)(n).
	\end{equation*}
	\begin{remark}
		The fractional powers for the discrete operators $\Delta_d$, $(-\Delta/\nabla)$ and $\Delta_{dd}$ have been obtained using the \emph{Balakrishnan's Formula}. The work \cite{GLM} a general representation for fractional powers of discrete operators which are generators of uniformly bounded $C_0$-semigroup in the Banach Algebra $\ell^1(\mathbb{Z})$ framework is presented.
	\end{remark}

	%\end{remark}
	\subsection{Markovian semigroups}
	
	The concept of Markovian semigroups is the key for the achievement of this work.
	\begin{definition}\label{Markovian}
		Let $(X,\Sigma,\mu)$ be a $\sigma$-finite measure space. We denote $M\subseteq \ell^1(X):=\ell^1(X,\Sigma,\mu)$ as the subset of all density functions over $X$
		\begin{equation*}
			M=\{f\in \ell^1(X):f\geq 0, \|f\|_1=1\}.
		\end{equation*}
		We call \emph{Markov Operator} to all linear mapping $F:\ell^1(X)\to\ell^1(X)$ satisfying $F(M)\subset M$.
	\end{definition}
	\begin{definition}
		A family $\{T(t)\}_{t\geq 0}$ of Markov operators is called \emph{Markovian semigroup} if furthermore it is a $C_0$-semigroup family of operators.
	\end{definition}
	\begin{remark}We note that if $A$ is a convolution operator with kernel $a\in\ell^1(\Z)$, with $e^{tA}$ a Markovian semigroup, then $e^{ta}$ satisfies $e^{ta}\geq 0$ and $\|e^{ta}\|_{1}=1$. Reciprocally, the assertion is true using \eqref{normsemigroupskernel}.
	\end{remark}
	\begin{remark}
		Let $\textbf{1}(n)$ the sequence constantly equals to $1$. Let $A$ be a convolution operator with kernel $a \in \ell^1(Z)$ be a generator of a Markovian semigroup $e^{tA}$. Then the following equality holds
		\begin{equation}\label{Markovianproperty}
			e^{tA}\textbf{1}(n)=\textbf{1}(n).
		\end{equation}
	\end{remark}
	\begin{proof} We remark that $\textbf{1}(n)=\textbf{1}(n-m)$, for all $n,m\in \Z$. Since $e^{tA}$ is a Markovian semigroup, we have
		\begin{align*}
			e^{tA}\textbf{1}(n)=(e^{ta}\ast \textbf{1})(n)=\sum_{m \in \Z}^{}e^{ta}(m)\textbf{1}(n-m)=\textbf{1}(n)\sum_{m \in \Z}^{}e^{ta}(m)=\textbf{1}(n).
		\end{align*}
	\end{proof}
	
	\subsection{Resolvent and Integral Resolvent families of Operators}\label{subordination}
	
	A powerful tool in order to study existence of solution of abstract evolution equations is the \emph{resolvent family of operators theory}, playing a crucial role in the study  of the well-posedness of fractional systems (see \cite{Ba2001, KLW1, Pruss}).
	
	\begin{definition} \label{def:2.6}
		Let $B$ be a closed linear operator with domain $D(B) \subset X$ and $a \in L^1_{loc}(\mathbb{R}^+_0).$ A family $\{S(t)\}_{t\geq 0}$ of bounded and linear operators in $X$ is called a \emph{resolvent} with generator $B$ if the following conditions are satisfied:
		\begin{enumerate}
			\item [(1)] $S(t)$ is strongly continuous on $\R_0^+$ and $S(0)=I$;
			\item [(2)] $S(t)$ commutes with $B$. That means that $S(t){D}(B)\subset {D}(B)$ and $BS(t)x=S(t)Bx$, for all $x \in {D}(B)$ and $t\geq 0$;
			\item [(3)] The resolvent equation holds:
			\begin{equation*}
				\label{S}
				S(t)x= x + \int_{0}^{t}a(t-s)BS(s)xds, \quad\mbox{for all}\ x \in {D}(B), \,\, t \geq 0.
			\end{equation*} 
		\end{enumerate}
	\end{definition}
	\begin{definition} \label{def:2.7}
		Let $B$ be a closed linear operator with domain $D(B) \subset X$ and $a \in C(\mathbb{R}^+_0).$ A strongly continuous family $\{P(t)\}_{t\geq 0}$ of bounded linear operators in $X$ is called an \emph{integral resolvent} with generator $B$ if the following conditions are satisfied:
		\begin{enumerate}
			\item [(1)] $P(0)= a(0)I;$
			\item [(2)] $P(t)$ commutes with $B;$
			\item [(3)] The integral resolvent equation holds:
			\begin{equation*}
				\label{P1}
				P(t)x= a(t)x + \int_{0}^{t}a(t-s)BP(s)xds, \quad\mbox{for all}\ x \in {D}(B), t \geq 0.
			\end{equation*} 
		\end{enumerate}
	\end{definition}
	
	\begin{remark}
		It is well known the fact for the case $a(t)\equiv 1$ the resolvent family corresponds to the $C_0-$semigroup generated by $A$ (see \cite{Pruss}).
	\end{remark}

	\section{Main Results}

	In this section we present the main purpose of this paper, which is to establish the well-posedness of the nonlinear model \eqref{Mainequation}. We recall the equation \eqref{Mainequation} is given by
	
	\begin{equation*}
		\begin{cases}  \mathbb D^\beta_t u(n,t)= B u(n,t) + f(n-ct,u(n,t)), &n\in\mathbb{Z}, \;t>0,\\
			u(n,0)=\varphi(n),\; &n\in\mathbb{Z},\\
		\end{cases}
	\end{equation*}
	where $\mathbb{D}_t^{\beta}$ is the fractional Caputo derivative, $\beta \in (0,1]$, $c>0$ some positive constant, the discrete convolution operator $Bh(n)=(b \ast h)(n)$, with $b$ belonging to Banach Algebra $\ell^{1}(\mathbb{Z})$ and $f,\varphi$ functions with suitable hypotheses.

	In order to prove the main result of this work, for each $\gamma>0$, we define the set
	\begin{equation*}
		\mathcal{L}_{\gamma}:=\{u\in \ell^{\infty}(\Z):0\leq u(n)\leq\gamma, \forall n \in \Z\}.
	\end{equation*}
	
	Let $\rho>0$. We define the function: $F:\mathbb{R}\times[0,\gamma]\to \mathbb{R}$, given by
	\begin{equation}\label{Definition F}
		F(x,s):=\rho s+ f(x,s).
	\end{equation}
	This function satisfies properties presented in the following Lemma. The proof can be found in \cite[Lemma 4.1]{KLW3}.
	\begin{lemma}\label{lemmaF}
		Let $\gamma>0$. Suppose that $f:\R\times\R_0^+$ is measurable, the restriction of $f(x,\cdot)$ to $[0,\gamma]$ belongs to $C^1[0,\gamma]$, is concave in $[0,\gamma]$, and satisfies $f(x,0)=0$, for all $x \in \R$. Let $\rho>0$ such that
		\begin{equation}\label{rhoinequality}
			\rho+\partial_sf(x,\gamma)\geq 0, \ \forall x \in \R.
		\end{equation}
		Then, the function $F$ defined in \eqref{Definition F} is no decreasing in the second variable and no negative for $x \in \R$.
	\end{lemma}

	We get the expression of the solution for the equation \eqref{Mainequation} employing the resolvent operator theory and the subordination principle on Wright function \eqref{Wright} presented in \cite{Ba2001, KLW1}. Let us define the following operators
	\begin{equation}\label{Resolvent}
		S_{\beta}(t):=\int_{0}^{t}\Phi_{\beta}(\tau)e^{\tau t^{\beta}B}d\tau, \ t\geq 0,
	\end{equation}
	called the \emph{resolvent family} for \eqref{Mainequation}, and
	\begin{equation}\label{IntegralResolvent}
		P_{\beta}(t):=\int_{0}^{t}\beta\tau\Phi_{\beta}(\tau)e^{\tau t^{\beta}B}d\tau, \ t\geq 0,
	\end{equation}
	called the \emph{integral resolvent family} for \eqref{Mainequation}. Then, the mild solution $u(n,t)$ for the nonlinear model \eqref{Mainequation} satisfies the equation
	\begin{equation}\label{Auxequation}
		u(n,t)=S_{\beta}(t)\varphi(n)+\int_{0}^{t}P_{\beta}(t-s)f(n-cs,u(n,s))\ ds.
	\end{equation}
	
	We consider the perturbed operator $B_{\rho}:=B-\rho I$. This operator is bounded over $\ell^p(\Z)$. Let us denote $\{e^{tB_{\rho}}\}_{t\geq0}$ the semigroup generated by $B_{\rho}$. From general theory of $C_0$-semigroup family of operators, we have $e^{tB_{\rho}}=e^{t(B-\rho I)}=e^{tB}e^{-\rho t}$. A change of variable of \eqref{Resolvent} and \eqref{IntegralResolvent} gives us an equivalent representation, given by
	\begin{equation}\label{Resolventchange}
		S_{\beta}(t):=\int_{0}^{t}t^{-\beta}\Phi_{\beta}(\tau t^{-\beta})e^{\tau B_{\rho}}d\tau, \ t\geq 0,
	\end{equation}
	and
	\begin{equation}\label{IntegralResolventchange}
		P_{\beta}(t):=\int_{0}^{t}\beta\frac{\tau}{t^{1+\beta}}\Phi_{\beta}(\tau t^{-\beta})e^{\tau B_{\rho}}d\tau, \ t\geq 0.
	\end{equation}
	
	\begin{proposition} Let $S_{\beta}(t)$ and $P_{\beta}$ be the resolvent and integral resolvent respectively generated by the operator $B$. Let $S_{\beta}^{\rho}(t)$ and $P_{\beta}^{\rho}(t)$ be the resolvent and integral resolvent respectively generated by the operator $-\rho I$. Then, we have the inequalities 
		\begin{equation}\label{cotaresolventeB}
			S_{\beta}(t)\varphi(n)\leq S_{\beta}^{\rho}(t)\varphi(n)
		\end{equation}
		and 
		\begin{equation}\label{cotaintegralresolvente}
			P_{\beta}(t)\varphi(n)\leq P_{\beta}^{\rho}(t)\varphi(n).
		\end{equation}
	\end{proposition}
	
	\begin{proof}
		We use the representation \eqref{Resolventchange} \eqref{IntegralResolventchange} and the fact that $B$ is the infinitesimal generator of a Markovian semigroup $e^{tB}$
		\begin{align*}
			S_{\beta}(t)\varphi(n)&=\int_{0}^{t}t^{-\beta}\Phi_{\beta}(\tau t^{-\beta})e^{\tau B_{\rho}}\varphi(n)d\tau\\
			&=\int_{0}^{t}t^{-\beta}\Phi_{\beta}(\tau t^{-\beta})e^{\tau B}e^{-t\rho}\varphi(n)d\tau\\
			&\leq \int_{0}^{t}t^{-\beta}\Phi_{\beta}(\tau t^{-\beta})e^{-t\rho}\varphi(n)d\tau\\
			& =S_{\beta}^{\rho}(t)\varphi(n),
		\end{align*} 
		Similarly for the integral resolvent, we consider any $\varphi\in \mathcal{L}_{\gamma}$, we get
		\begin{align*}
			P_{\beta}(t)\varphi(n)&=\int_{0}^{t}\beta\frac{\tau}{t^{1+\beta}}\Phi_{\beta}(\tau t^{-\beta})e^{\tau B_{\rho}}\varphi(n)d\tau\\
			&=\int_{0}^{t}\beta\frac{\tau}{t^{1+\beta}}\Phi_{\beta}(\tau t^{-\beta})e^{\tau B}e^{-t\rho}\varphi(n)d\tau\\
			&\leq \int_{0}^{t}\beta\frac{\tau}{t^{1+\beta}}\Phi_{\beta}(\tau t^{-\beta})e^{-t\rho}\varphi(n)d\tau\\
			& =P_{\beta}^{\rho}(t)\varphi(n).
		\end{align*}
		
	\end{proof}

	Next, we present the main result of this paper.
	\begin{theorem}\label{MainTheorem}
		Let $\gamma>0$. Suppose the hypotheses of Lemma \ref{lemmaF} holds and $f(x,\gamma)\leq 0$, for all $x \in \R$. Furthermore, suppose that $\partial_sf(x,0)$ is a non-increasing function of $x$. Let the initial value $\varphi \in \mathcal{L}_{\gamma}$. Then, there exists a unique mild solution $u \in C(\R_0^+,\mathcal{L}_{\gamma})$ to the equation \eqref{Auxequation} and hence a unique solution for nonlinear system \eqref{Mainequation}.
	\end{theorem}
	\begin{proof}
		We used a iterative method. We define the following operator
		\begin{equation}\label{fixedpointoperator}
			\mathcal{K}_{\beta}u(n,t):= S_{\beta}(t)\varphi(n)+\int_{0}^{t}P_{\beta}(t-s)F(n-cs,u(n,s))\ ds,
		\end{equation}
		where $F(x,y)=\rho y+ f(x,y)$ and inequality \eqref{rhoinequality} holds. Note that $F(x,0)=f(x,0)=0$. Furthermore, its follows from Dominated convergence Lebesgue Theorem that the operator $\mathcal{K}_{\beta}$ is continuous over $C(\R^+,\mathcal{L}_{\gamma})$.
		
		We construct the sequence 
		\begin{equation*}\label{lower}
			v_{k}(n,t)=\begin{cases}
				\mathcal{K}_{\beta}v_{k-1}(n,t),& \ \mbox{if}\ k\geq 1,\\
				0,& \ \mbox{if}\ k=0.
			\end{cases}
		\end{equation*}
		
		We computing $v_1(n,t)$, obtaining
		\begin{equation*}
			v_1(n,t)= S_{\beta}(t)\varphi(n)+\int_{0}^{t}P_{\beta}(t-s)F(n-cs,0)\ ds=  S_{\beta}(t)\varphi(n)\geq 0,
		\end{equation*}
		since $S_{\beta}(t)$ is a positive operator. Using the fact that $F(x,y)$ is no negative for all $x \in \R$ and $0\leq y\leq\gamma$, we get
		\begin{align*}
			v_1(n,t)&= S_{\beta}(t)\varphi(n)\\
			& \leq S_{\beta}(t)\varphi(n)+\int_{0}^{t}P_{\beta}(t-s)F(n-cs,v_1(n,t))\ ds\\
			&=v_2(n,t).
		\end{align*} 
		We remark that $F(x,y)$ is no decreasing in the second variable by Lemma \eqref{lemmaF}. Let $k \in \N$. We shall prove that $v_k(n,t)\leq v_{k+1}(n,t)$. Indeed, we have
		\begin{align*}
			v_k(n,t)&= S_{\beta}(t)\varphi(n)+\int_{0}^{t}P_{\beta}(t-s)F(n-cs,v_k(n,t))\ ds\\
			& \leq S_{\beta}(t)\varphi(n)+\int_{0}^{t}P_{\beta}(t-s)F(n-cs,v_{k+1}(n,t))\ ds\\
			&=v_{k+1}(n,t).
		\end{align*}
		Therefore, we get the inequalities $v_0\leq v_1 \leq \ldots \leq v_k\leq v_{k+1}\leq \ldots$.
		
		On the other hand, we consider the sequence
		\begin{equation*}\label{upper}
			w_{k}(n,t)=\begin{cases}
				\mathcal{K}_{\beta}w_{k-1}(n,t),& \ \mbox{if}\ k\geq 1,\\
				\gamma,& \ \mbox{if}\ k=0.
			\end{cases}
		\end{equation*}
		Then, using  that $F(x,y)$ is no decreasing in the second variable, we get
		\begin{align*}
			w_1(n,t)&= S_{\beta}(t)\varphi(n)+\int_{0}^{t}P_{\beta}(t-s)F(n-cs,w_0(n,t))\ ds\\
			& = S_{\beta}(t)\varphi(n)+\int_{0}^{t}P_{\beta}(t-s)F(n-cs,\gamma)\ ds.\\
		\end{align*}
		Note the fact that $0\leq F(x,y) \leq \rho\gamma$ since $f(x,\gamma)\leq 0$. Using this fact and the inequalities \eqref{cotaresolventeB} and \eqref{cotaintegralresolvente} we obtain
		\begin{align*}
			w_1(n,t)&\leq S_{\beta}^{\rho}(t)\varphi(n)+\int_{0}^{t}P_{\beta}^{\rho}(t-s)F(n-cs,\gamma)\ ds\\
			& \leq S_{\beta}^{\rho}(t)\varphi(n)+\rho\gamma\int_{0}^{t}P_{\beta}^{\rho}(t-s)\ ds.
		\end{align*}
		We recall that $0\leq \varphi(n)\leq \gamma$ and the fact that $S_{\beta}^{\rho}(0)=1$. Furthermore, we recall that
		\begin{equation*}\label{relationresolventeintresolvente}
			\frac{d}{dt}S_{\beta}^{\rho}(t)=-\rho P_{\beta}^{\rho}(t).
		\end{equation*}
		Therefore, integrating the above equality and computing, we have
		\begin{align*}\label{inequalityw1}
			w_1(n,t)& \leq S_{\beta}^{\rho}(t)\varphi(n)+\rho\gamma\int_{0}^{t}P_{\beta}^{\rho}(t-s)\ ds\\
			&  \leq S_{\beta}^{\rho}(t)\gamma - S_{\beta}^{\rho}(t)\gamma+\gamma\\
			&=\gamma.
		\end{align*}
		We use above inequality and Lemma \eqref{lemmaF} for to obtain $$F(n-cs,w_1(n,t))\leq F(n-cs,\gamma).$$ Hence, we obtain
		\begin{align*}
			w_1(n,t)& =S_{\beta}(t)\varphi(n)+\int_{0}^{t}P_{\beta}(t-s)F(n-cs,\gamma)\ ds\\
			&\leq  S_{\beta}(t)\varphi(n)+\int_{0}^{t}P_{\beta}(t-s)F(n-cs,w_1)\ ds \\
			&=w_2(n,t).
		\end{align*}
		Let $k \in \N$. We shall prove that $w_{k+1}(n,t)\leq w_{k}(n,t)$. Indeed, we have
		\begin{align*}
			w_{k+1}(n,t)& =S_{\beta}(t)\varphi(n)+\int_{0}^{t}P_{\beta}(t-s)F(n-cs,w_k(n,t))\ ds\\
			&\leq  S_{\beta}(t)\varphi(n)+\int_{0}^{t}P_{\beta}(t-s)F(n-cs,w_{k-1})\ ds \\
			&=w_k(n,t).
		\end{align*}
		Therefore, we get
		\begin{equation*}
			\gamma\geq w_1(n,t)\geq w_2(n,t)\geq\ldots\geq w_k(n,t)\geq w_{k+1}(n,t)\geq \ldots.
		\end{equation*}
		\underline{Assertion}: For any $j,k \in \N_0$, we have
		\begin{equation*}
			0\leq v_1(n,t)\leq\ldots\leq v_k(n,t)\leq \ldots\leq w_j(n,t)\leq\ldots \leq\w_1(n,t)\leq\gamma.
		\end{equation*}
		We suppose that there exists $j,k$ such that $w_j(n,t)\leq v_k(n,t)$. Then, we can find the first entire number $m \in \N$ such that $w_m(n,t)> v_m(n,t)$. Using lemma \eqref{lemmaF} we get
		\begin{align*}
			0<(v_m-w_m)&=\int_{0}^{t}P_{\beta}(t-s)[F(n-cs,v_{m-1})-F(n-cs,w_{m-1})]\leq 0,
		\end{align*}
		which is a contradiction.
		
		Therefore, it is clear that $\{v_k\}_{k\in\N}$ and $\{w_k\}_{k\in\N}$ are convergent sequences. Let us define $\displaystyle v:=\lim_{k\to \infty}v_k$ and $\displaystyle w:=\lim_{k\to \infty}w_k$. It is clear from assertion that $v(n,t)\leq w(n,t)$. Moreover, we have that $v,w\in C(\R^+,\mathcal{L}_{\gamma})$. Since the operator $K_{\beta}$ is continuous over $C(\R^+,\mathcal{L}_{\gamma})$, we have that $v$ and $w$ are solution for \eqref{Auxequation}.
		
		We shall prove uniqueness of solution. Indeed, the application the Mean Value Theorem yields
		\begin{align*}
			0\leq(w-v)(n,t)&=\int_{0}^{t}P_{\beta}(t-s)[F(n-cs,w(n,t))-F(n-cs,v(n,t))]\\
			&=\int_{0}^{t}P_{\beta}(t-s)\partial_s F(n-cs,\theta)[(w-v)(n,t)],
		\end{align*}
		where $v(n,t)\leq\theta\leq w(n,t)$. Since $\partial_s F(n-cs,\theta)=\rho+\partial_s f(n-cs,\theta)$, the Lemma \eqref{lemmaF} yields $\partial_s f(n-cs,\theta)$ is non-increasing and hence $\partial_s f(n-cs,\gamma)\leq\partial_s f(n-cs,\theta)\leq\partial_s f(n-cs,0)$. Consequently, we have $\leq\partial_s F(n-cs,\theta)\leq\partial_s F(n-cs,0)$.
		
		Observe that, for $y$ fixed, $\partial_sF(\cdot,y)$ is no decreasing. Then, from above inequality, we obtain
		\begin{align*}
			0\leq(w-v)(n,t)&=\int_{0}^{t}P_{\beta}(t-s)[F(n-cs,w(n,t))-F(n-cs,v(n,t))]\\
			&=\partial_s F(n,0)\int_{0}^{t}P_{\beta}(t-s)[(w-v)(n,t)].
		\end{align*}
		Therefore, we get
		\begin{align*}
			0&\leq \sup_{n\in\Z_+}[ w(n,t)-v(n,t)]\\
			& \leq \partial_s F(n,0)\int_{0}^{t}P_{\beta}(t-s) \sup_{n\in\Z_+}[(w-v)(n,t)]\\
			& \leq \partial_s F(n,0)\int_{0}^{t}P_{\beta}^{\rho}(t-s) \sup_{n\in\Z_+}[(w-v)(n,t)],
		\end{align*}
		where we used that \eqref{cotaintegralresolvente}. Finally, from Gronwall inequality (see e.g \cite{Ye}) for integral inequalities, we obtain $v\equiv w$, getting the desired result. 
	\end{proof}
	
	Next, we present two consequences of the Theorem \eqref{MainTheorem} concerning to a comparison principle between solution with respective initial values.
	
	\begin{corollary}\label{Corollarycomparison1}
		Assume the hypotheses of Theorem \eqref{MainTheorem}. Let $u_1(n,t)$ and $u_2(n,t)$ the solutions for equation \eqref{Auxequation} with initial condition $\varphi$ and $\psi$ respectively. If $\varphi(n)\leq\psi(n)$, for all $n \in \Z$, then $u_1(n,t)\leq u_2(n,t)$, for all $n \in \Z$ and $t \geq 0$.
	\end{corollary}
	
	\begin{proof}
		Analogously to proof of the Theorem \eqref{MainTheorem}, we consider the sequence $\{u_{1,k}\}_{k\in \N}$ and $\{u_{2,k}\}_{k\in \N}$, defined according to iterative definition \eqref{lower}. By uniqueness of solution for the system \eqref{Auxequation}, we have $u_{1,k}\to u_1$ and $u_{2,k}\to u_2$ when $n \to \infty$. Since $u_{1,0}= u_{2,0}=0$ and $F(x,0)=0$, we have
		\begin{align*}
			u_{1,1}(n,t)&=S_{\beta}(t)\varphi(n)+\int_{0}^{t}P_{\beta}(t-s)F(n-cs,u_{1,0}(n,t))\ ds\\
			&=S_{\beta}(t)\varphi(n)
		\end{align*}
		and
		\begin{align*}
			u_{2,1}(n,t)&=S_{\beta}(t)\varphi(n)+\int_{0}^{t}P_{\beta}(t-s)F(n-cs,u_{2,0}(n,t))\ ds\\
			&=S_{\beta}(t)\psi(n).
		\end{align*}
		Since $S_{\beta}(t)$ is a positive operator, we have $u_{1,1}(n,t)\leq u_{2,1}(n,t)$, for all $n \in \Z$ and $t\geq0$.
		
		Using Lemma \eqref{lemmaF}, we get
		\begin{align*}
			u_{1,2}(n,t)&=S_{\beta}(t)\varphi(n)+\int_{0}^{t}P_{\beta}(t-s)F(n-cs,u_{1,1}(n,t))\ ds\\
			&=S_{\beta}(t)\varphi(n)+\int_{0}^{t}P_{\beta}(t-s)F(n-cs,u_{2,1}(n,t))\ ds\\
			&=u_{2,2}(n,t).
		\end{align*}
		By induction principle, we obtain $u_{1,k}(n,t)\leq u_{2,k}(n,t)$ for all $k \in \N$ and therefore, $u_1(n,t)\leq u_2(n,t)$. The proof is finished.
	\end{proof}
	
	\begin{corollary}
		Assume the hypotheses of Theorem \eqref{MainTheorem}. Let $v,w\in C(\R_0^+,\mathcal{L}_{\gamma})$ such that $\mathcal{K}_{\beta}v(n,t)\geq v(n,t)$ and $\mathcal{K}_{\beta}w(n,t)\leq w(n,t)$, for all $n\in\Z$ and $t\in \mathbb{R}_0^+$. Suppose that $v(n,0)\leq w(n,0)$ for all $n\in\Z$. Then the inequality $v(n,t)\leq w(n,t)$ for all $n\in\Z$ and $t\in \R_0^+$ holds.
	\end{corollary}
	\begin{proof}
		We use the Lemma \eqref{lemmaF} in order to obtain
		\begin{align*}
			\mathcal{K}_{\beta}^2v(n,t)&=S_{\beta}(t)\varphi(n)+\int_{0}^{t}P_{\beta}(t-s)F(n-cs,\mathcal{K}_{\beta}v(n,t))\ ds\\
			&\geq S_{\beta}(t)\varphi(n)+\int_{0}^{t}P_{\beta}(t-s)F(n-cs,v(n,t))\ ds\\
			&=\mathcal{K}_{\beta}v(n,t).
		\end{align*}
		Applying recursively the above process, we get
		\begin{equation*}
			v(n,t)\leq \mathcal{K}_{\beta}v(n,t)\leq\mathcal{K}_{\beta}^2v(n,t)\leq\ldots\leq\lim_{k\to \infty}\mathcal{K}_{\beta}^kv(n,t):=v^*(n,t).
		\end{equation*} 
		Analogously, from Lemma \eqref{lemmaF}, we obtain
		\begin{align*}
			\mathcal{K}_{\beta}^2w(n,t)&=S_{\beta}(t)\varphi(n)+\int_{0}^{t}P_{\beta}(t-s)F(n-cs,\mathcal{K}_{\beta}w(n,t))\ ds\\
			&\leq S_{\beta}(t)\varphi(n)+\int_{0}^{t}P_{\beta}(t-s)F(n-cs,w(n,t))\ ds\\
			&=\mathcal{K}_{\beta}w(n,t).
		\end{align*}
		and proceeding inductively, we get 
		\begin{equation*}
			w(n,t)\geq \mathcal{K}_{\beta}w(n,t)\geq\mathcal{K}_{\beta}^2w(n,t)\geq\ldots\geq\lim_{k\to \infty}\mathcal{K}_{\beta}^kw(n,t):=w^*(n,t).
		\end{equation*}
		Note that $v^*$ and $w^*$ there exists, since the operator $\mathcal{K}_{\beta}$ is continuous.
		Furthermore, we remark that $v^*$ and $w^*$ are solution for the equation \eqref{Auxequation} with initial values $v(n,0)$ and $w(n,0)$ respectively. Then, the Corollary \eqref{Corollarycomparison1} yields
		\begin{equation*}
			v(n,t)\leq v^*(n,t)\leq w^*(n,t)\leq w(n,t),
		\end{equation*}
		for all $n \in \Z$ and $t\geq 0$. The proof is finished.
	\end{proof}
	
	\section{Applications and Examples}
	In this section we present convolution operators and nonlinear functions satisfying the hypotheses of Theorem \eqref{MainTheorem}. We conclude the paper with a historical review about some cases and its applications in diverse models.
	\subsection{Convolution Operators}
	
	We present a sufficient condition in order to guarantee the Markov properties for the $C_0$-semigroup generated by the fractional powers of a discrete convolution operator generating a Markovian $C_0$-semigroup.
	
	\begin{theorem}\label{Markovianfractionalpowers} Let $A \in \mathcal{B}(\ell^p(\Z))$ be an operator by discrete convolution, namely $Af(n)=(a\ast f)(n)$, which is the infinitesimal generator of Markovian semigroup $e^{tA}$ and $\alpha \in (0,1)$. Then, the semigroup generated for the fractional powers $-(-A)^{\alpha}$, namely $e^{-t(-A)^{\alpha}}$ is a Markovian semigroup.
	\end{theorem}
	\begin{proof}
		The key of the proof is the L\'evy subordination principle and its properties.
		Let $\phi\geq 0$ be a sequence in $\ell^p(\Z)$, $1\leq p \leq \infty$. Since $\alpha \in (0,1)$, from remark \eqref{remlevy} property i) we have
		\begin{equation*}
			e^{-t(-A)^{\alpha}}\phi(n)=\int_0^{\infty} e^{\lambda A}\phi(n) f_{t,\alpha}(\lambda) d\lambda, \quad t>0.
		\end{equation*}
		The positivity of L\'evy process (see remark \eqref{remlevy}, property ii)) and $e^{tA}$ (Markovian property) yield
		\begin{equation*}
			e^{-t(-A)^{\alpha}}\phi(n)\geq 0.
		\end{equation*}
		Let us choose $\phi(n)=\delta_{0}(n)$. A straightforward computation yield
		\begin{equation*}
			0\leq \int_0^{\infty} e^{\lambda A}\delta_{0}(n) f_{t,\alpha}(\lambda) d\lambda=e^{-t(-A)^{\alpha}}  \delta_{0}(n)= e^{-t(-a)^{\alpha}}.
		\end{equation*}
		
		Let $\textbf{1}(n)$ be the sequence constantly equals to $1$ considered in \eqref{Markovianproperty}. We use again the integral representation of L\'evy subordination principle, and \eqref{Markovianproperty} in order to obtain
		\begin{equation*}
			e^{-t(-A)^{\alpha}}\textbf{1}(n)=\int_0^{\infty} e^{\lambda A}\textbf{1}(n) f_{t,\alpha}(\lambda) d\lambda=\textbf{1}(n)\int_0^{\infty}  f_{t,\alpha}(\lambda) d\lambda=\textbf{1}(n).
		\end{equation*}
		On the other hand, since $e^{-t(-A)^{\alpha}}$ is defined by discrete convolution, we have
		\begin{equation*}
			e^{-t(-A)^{\alpha}}\textbf{1}(n)=\sum_{m \in \Z}^{}e^{-t(-a)^{\alpha}}(m)\textbf{1}(n-m)=\textbf{1}(n)\sum_{m \in \Z}^{}e^{-t(-a)^{\alpha}}(m).
		\end{equation*}
		Therefore, from above equalities, we have $\displaystyle\|e^{-(-a)^{\alpha}}\|_1= \sum_{n \in \Z}^{}e^{-t(-a)^{\alpha}}(n)=1$. From \eqref{normsemigroupskernel}, we obtain $\|e^{-(-A)^{\alpha}}\|_1=1$. The proof is finished.
		
	\end{proof}
	
	\begin{corollary} Let $\alpha \in (0,1)$. The semigroups $e^{-t(-\Delta_d)^{\alpha}}$, $e^{-t(\Delta/-\nabla)^{\alpha}}$ and $e^{-t(-\Delta_{dd})^{\alpha}}$ are Markovian.
	\end{corollary}
	\begin{proof}
		A straightforward computation yields that $e^{t\Delta_d}$ and $e^{t(-\Delta/\nabla)}$ are Markovian semigroups. Indeed, from generatrix formula for $I_{\nu}$ Bessel function, we have
		\begin{equation*}
			\|e^{t\Delta_d}\|_1=e^{-2t}\sum_{n \in \mathbb{Z}}^{}I_n(2t)=1
		\end{equation*}
		and
		\begin{equation*}
			\|e^{t\Delta_{dd}}\|_1=e^{-2t}\sum_{n\in \mathbb{Z}}^{}I_\frac{n}{2}\chi_{2\mathbb{Z}} (2t)=1.
		\end{equation*}
		Furthermore
		\begin{equation*}
			\|e^{t(-\Delta/\nabla)}\|_1=e^{-t}\sum_{n \in \mathbb{Z}}^{}\frac{t^{(\pm n)}}{(\pm n)!}\chi_{\mathbb{Z}_{\pm}}(n)=1.
		\end{equation*}
		The positive condition is immediate from positivity of $I_n(2t), e^{-2t}$ and $t^n$, since $t\geq 0$.
		
		Therefore, from the above proposition, we obtain the desired result.
	\end{proof}
	
	\begin{remark}
		The above result extends the result obtained by Lizama and Roncal in \cite{LiRo}, Theorem 1.3, v) in the study of the discrete fractional Laplacian and the properties of the generated semigroup. 
	\end{remark}
	
	\subsection{Nonlinear Functions}
	
	Next, we present nonlinear functions which satisfies the hypotheses of Theorem \eqref{MainTheorem}.
	\begin{enumerate}
		\item Let $a \in \R^+$ a fixed constant and define
		\begin{equation}\label{Nonlinear1}
			f(s)=s(a^2-s^2).
		\end{equation}
		It is clear that $f$ is a measurable function, the restriction $f$ belongs to $C^1(0,a)$ and $f(0)=f(a)=0$. Its partial derivatives $\partial_sf(s)=a^2-3s^2$  and $\partial_s^2f(s)=-6s<0$ gives us the concavity of $f$ on $[0,a]$. Let us choose $\rho=3a^2$. With this choice, we ensure that the inequality \eqref{rhoinequality} holds. Furthermore, its clear that $\partial_sf(x,0)$ is a no decreasing function, and $f(x,\gamma)=\gamma(r(x)-\gamma^p)=\gamma^{\frac{1}{p}}(r(x)-\gamma)\leq 0$. Therefore, the function $f$ defined in \eqref{Nonlinear1} satisfies all the hypotheses of the Theorem \eqref{MainTheorem}.
		
		\item Let $r:\R\to\R$ be continuous, no decreasing, bounded and piecewise continuously differentiable satisfying $0<r(\infty)<\infty$. In fact, we define $r(\infty)=\gamma$.
		We define the nonlinear function $f:\R\times\R^+\to\R$ as follows
		\begin{equation}\label{Nonlinear2}
			f(x,s)=s(r(x)-s).
		\end{equation}
		It is clear that $f$ is a measurable function, the restriction $f(x,\cdot)$ belongs to $C^1(0,\gamma)$ and $f(x,0)=0$, for all $x\in \R$. Its partial derivatives $\partial_sf(x,s)=r(x)-2s$  and $\partial_s^2f(x,s)=-2<0$ gives us the concavity of $f$ on $[0,\gamma]$. Let us choose $\rho=3\gamma$. With this choice, we ensure that the inequality \eqref{rhoinequality} holds. Furthermore, its clear that $\partial_sf(x,0)$ is a no decreasing function, and $f(x,\gamma)=\gamma(r(x)-\gamma)\leq 0$. Therefore, the function $f$ defined in \eqref{Nonlinear2} satisfies all the hypotheses of the Theorem \eqref{MainTheorem}.
		
		\item We can generalize the above function. Let $r:\R\to\R$ be continuous, no decreasing, bounded and piecewise continuously differentiable satisfying $0<r(\infty)<\infty$. In fact, we define $r(\infty)=\gamma^{\frac{1}{p}}$.
		Let $p\geq 1$. We define the nonlinear function $f:\R\times\R^+\to\R$ as follows
		\begin{equation}\label{Nonlinear3}
			f(x,s)=s(r(x)-s^p).
		\end{equation}
		It is clear that $f$ is a measurable function, the restriction $f(x,\cdot)$ belongs to $C^1(0,\gamma)$ and $f(x,0)=0$, for all $x\in \R$. Its partial derivatives $\partial_sf(x,s)=r(x)-(p+1)s^p$  and $\partial_s^2f(x,s)=-p(p+1)s^{p-1}<0$ gives us the concavity of $f$ on $[0,\gamma]$. Let us choose $\rho=(p+2)\gamma^p$. With this choice, we ensure that the inequality \eqref{rhoinequality} holds. Furthermore, its clear that $\partial_sf(x,0)$ is a no decreasing function, and $f(x,\gamma)=\gamma(r(x)-\gamma^p)=\gamma^{\frac{1}{p}}(r(x)-\gamma)\leq 0$. Therefore, the function $f$ defined in \eqref{Nonlinear3} satisfies all the hypotheses of the Theorem \eqref{MainTheorem}.
	\end{enumerate}
	
	\begin{remark}The model \eqref{Mainequation}, considering the fractional Laplacian and the nonlinear function \eqref{Nonlinear1} is the well-known \emph{discrete fractional generalized Fisher-KPP equation}. This includes the discrete generalized Fisher-KPP equation with delay since $f(n-ct,s)= s(r(n-ct)-s)$ with $n \in \Z$ and some constant $c>0$. The classical Fisher-KPP equation corresponds to the case where $r(x)\equiv\gamma$.
		
		When we considerate $r\equiv1$ and $p=2$, we recover the \emph{Newell-Whitehead-Segel equation}, which describes the so called Rayleigh-Benard convection. The Newell-Whitehead-Segel equation is a well-known universal equation to govern evolution of nearly one-dimensional nonlinear patterns produced by a finite-wavelength instability in isotropic two-dimensional media. More references for these applications, see \cite{Newell} and \cite{Segel}.
	\end{remark}
	
	\section[.]{Appendix}
	
	In this appendix we give the basic concepts of the special functions involved in the development of this study.
	
	\subsection{Wright function $\Phi_{\gamma}$}
	
	The Wright type function with one parameter from \cite[Formula (28)]{Go-Lu-Ma99} (see also \cite{Po1999,SKM1993,Wri}) is given by
	
	\begin{equation}\label{Wright}
		\Phi_{\alpha}(z):= \sum_{n=0}^{\infty} \frac{(-z)^n}{n! \Gamma(-\alpha n +1 -\alpha)} = \frac{1}{2\pi i}
		\int_{\gamma} \mu^{\alpha-1} e^{\mu - z \mu^{\alpha}} d\mu, \quad
		0< \alpha <1,
	\end{equation}
	where $\gamma$ is a contour which starts and ends
	at $-\infty$ and encircles the origin once counterclockwise. This special case has sometimes been called \emph{Mainardi} function.

	\begin{remark}
		{\em
			Let $z\in\mathbb C$, $t>0$ and $0<\alpha,\gamma<1$. Then the following properties hold:
			\begin{itemize}\label{RemarkWright}
				\item [(i)] $\displaystyle E_{\gamma,1}(z)=\int_{0}^{\infty}\Phi_{\gamma}(t)e^{zt}dt$.
				\item [(ii)] $\displaystyle\Phi_{\alpha}(t) \geq 0$.
				\item [(iii)] $\displaystyle\int_{0}^{\infty}\Phi_{\alpha}(t)dt =1$.
			\end{itemize}
		}
	\end{remark}
	
	It follows from (ii) and (iii) that $\Phi_{\alpha}$ is a probability density function on $\mathbb{R}_0^+$.
	Actually, the Wright function has been used for models in stochastic processes \cite{Go-Lu-Ma99, Go-Ma00}.

	\subsection{L\'{e}vy process $f_{t,\alpha}$}
	
	We present the following function, called stable L\'{e}vy process, defined for $0<\alpha<1$ by
	\begin{equation}\label{levyfunct}
		f_{t,\alpha}(\lambda) = \begin{cases}
			\displaystyle \frac{1}{2\pi i} \int_{\sigma - i \infty}^{ \sigma + i \infty} e^{ z \lambda - t z^{\alpha}} dz,  \;\;\;\; \sigma >0, \quad t>0,  &\lambda \geq 0, \\
			0  \quad &\lambda <0,
		\end{cases}
	\end{equation}
	where the branch of $z^{\alpha}$ is taken so that $\mbox{Re}(z^{\alpha})>0$ for $\mbox{Re}(z)>0.$ This branch is  single-valued in the
	$z$-plane cut along the negative real axis.
	These functions were introduced by S. Bochner \cite{Bochner1} in the study of certain stochastic processes. K. Yosida \cite{Yo80} used them systematically in the study of $C_0$-semigroups generated by fractional powers of uniformly bounded $C_0$-semigroups of linear operators.
	The L\'{e}vy functions are the density functions associated with the stable L\'evy processes in
	the rotational invariant case, and are related to the fractional Brownian motion.
	
	\begin{remark}\label{remlevy}
		{\em
			The following properties hold:
			\begin{itemize}
				\item[(i)] $\displaystyle \int_0^{\infty} e^{-\lambda a} f_{t,\alpha}(\lambda) d\lambda = e^{-ta^{\alpha}}, \quad t>0, \quad a>0,\quad  0<\alpha<1.$
				\item[(ii)] $\displaystyle f_{t,\alpha}(\lambda) \geq 0, \quad \lambda >0, \,\, t>0, \quad  0<\alpha<1.$
				\item[(iii)] $\displaystyle \int_0^{\infty}  f_{t,\alpha}(\lambda) d\lambda = 1, \quad t>0, \quad  0<\alpha<1.$
			\end{itemize}
			For a proof of (i)-(iii), see \cite[p.260-262]{Yo80}. 
		}
	\end{remark}

	\subsection{Bessel function $I_\nu$}
	For $\nu \in \mathbb{R}$, the \emph{Modified Bessel functions of the first kind} is defined by
	\begin{equation}\label{IBessel}
		I_{\nu}(x)=\sum_{n=0}^{\infty}\frac{1}{\Gamma(n+\nu+1)n!}\left(\frac{x}{2}\right)^{2n+\nu}.
	\end{equation}
	
	It is direct from the definition \eqref{IBessel} that $I_n(x) \geq 0, \quad n\in \mathbb{Z}, \,\, x\geq 0$,
	Following propierties can be founded in \cite[Formula 8.511]{GR}, \cite{AAR2006}.
	\begin{enumerate}
		\item $\displaystyle \sum_{n\in \mathbb{Z}} I_n(x)z^{n} = \displaystyle e^{\frac{x}{2}(z+\frac{1}{z})}, \quad z \in \mathbb{C}\setminus\{0\}$.
		\item $I_{-n}(x)=I_{n}(x)=(-1)^nI_n(-x)$.
		\item $\displaystyle I_n(x+y)=\sum_{k\in \mathbb{Z}}I_{n-k}(x)I_k(y)$.
	\end{enumerate}

\end{document}